\documentclass[12pt]{amsart}

\usepackage{amssymb}

\usepackage{graphicx}

\usepackage[cmtip,all]{xy}

\newtheorem{theorem}{Theorem}[section]

\newtheorem{proposition}[theorem]{Proposition}
\newtheorem{corollary}[theorem]{Corollary}
\theoremstyle{definition}
\newtheorem{definition}[theorem]{Definition}

\theoremstyle{remark}

\numberwithin{equation}{section}

\begin{document}

 \title[Type 1 subdiagonal algebras]{ Maximality and finiteness of type 1 subdiagonal algebras}


\author{Guoxing Ji}
\address{School  of Mathematics and Statistics,
  Shaanxi Normal University,
  Xian , 710119, People's  Republic of  China}
\email{gxji@snnu.edu.cn}
\thanks{This research was
supported by the National Natural
   Science Foundation of China(No. 11771261) and the Fundamental Research Funds for the Central Universities (Grant No. GK201801011)}



\date{}

\dedicatory{}


\begin{abstract}Let $\mathfrak A$ be a type 1 subdiagonal algebra in a $\sigma$-finite von Neumann algebra $\mathcal M$ with respect to a faithful normal conditional  expectation $\Phi$. We  give necessary and sufficient  conditions for  which $\mathfrak A$ is maximal among the $\sigma$-weakly closed subalgebras of $\mathcal M$.  In addition, we show that a  type 1 subdiagonal algebra in a finite von Neumann algebra is automatically finite which gives a positive answer of Arveson's finiteness problem in 1967 in type 1 case.
\end{abstract}

\maketitle


\bibliographystyle{amsplain}



\baselineskip18pt
\section{Introduction}
There are fruitful theorems in classical Hardy space theory. For example,  A  well-known classical result  on     bounded analytic function algebra $H^{\infty}(\mathbb T)$  is that it is   maximal           as a $w^*$-closed subalgebras in $L^{\infty}(\mathbb T)$.
       A noncommutative analogue is obtained by replacing $L^{\infty}(\mathbb T)$
by a von Neumann algebra $\mathcal M$  and $H^{\infty}(\mathbb T)$ by a unital $\sigma$-weakly closed subalgebra $\mathfrak A$ of $\mathcal M$.  There are many  successful noncommutative extensions of classical $H^p$ space theory from now on. One very important notion is subdiagonal algebras introduced by  Arveson in \cite{arv1}.       Based on  subdiagonal algebras, noncommutative Hardy spaces are developed(cf. \cite{ble,ble1,ble2,ble3,jig,jig1,jig2,jio,jis}).
For example, Marsalli and West \cite{mar2} gave a  Riesz factorization theorem  for finite  noncommutative $H^p$ spaces.
  Blecher and    Labuschagne established   Beurling type invariant subspace  theorems for a  finite subdiagonal algebra in \cite{ble4} and  Labuschagne in \cite{lab} extended  their results to noncommutative $H^2$ for  maximal subdiagonal algebras in a $\sigma$-finite von Neumann algebra. 
 On the other hand, several authors are of interest in  the maximality of an analytic  subalgebra  as a $\sigma$-weakly
  closed subalgebra in a von Neumann algebra.
 McAsey, Muhly and Saito(\cite{mms,mms1}) considered the maximality of  analytic crossed products. A necessary and sufficient condition for the maximality of analytic operator algebra $H^{\infty}(\alpha)$ determined by a flow $\alpha$ on a von Neumann algebra $\mathcal M$ is given by Solel in \cite{sol}. Very recently, Peligrad gave a complete solution  of the maximality problem for one-parameter dynamical systems in \cite{pel}.    Subdiagonal algebras are very  important classes of analytic operator algebras. It becomes interesting to determine the maximality of these algebras.
   We consider  related problems in this paper.  We
   firstly
   recall some notions.

  Let $\mathcal M$ be
a $\sigma$-finite von Neumann algebra acting on a complex  Hilbert
$\mathcal H$. $Z(\mathcal M)=\mathcal M\cap\mathcal M^{\prime}$ is the center of $\mathcal M$, where $\mathcal M^{\prime}$ is the commutant of $\mathcal M$. If $Z(\mathcal M)=\mathbb C I$,
the multiples of identity $I$, then $\mathcal M$ is said  to be a factor.  We denote by $\mathcal M_*$ the space of all
$\sigma$-weakly continuous linear functionals of $\mathcal M$. Let
$\Phi$  be  a faithful normal conditional expectation from $\mathcal
M$ onto  a von Neumann subalgebra   $\mathfrak D$. Arveson \cite{arv1} gave the following definition. A subalgebra
$\mathfrak A$ of $\mathcal M$, containing $\mathfrak D$, is called a
subdiagonal algebra of $\mathcal M$ with respect to $\Phi$ if

(i) $\mathfrak A \cap \mathfrak A^{*} =\mathfrak D$,

(ii) $\Phi$ is multiplicative on $\mathfrak A$, and

(iii)  $\mathfrak A + \mathfrak A^*$ is $\sigma$-weakly dense in
$\mathcal M$.

\noindent The algebra $\mathfrak D$ is called the diagonal of
$\mathfrak A$.
  we  may assume that  subdiagonal
 algebras are  $\sigma$-weakly closed without loss generality(cf.\cite{arv1}).

  We say that
$\mathfrak A$ is a maximal subdiagonal algebra in $\mathcal M$ with
respect to $\Phi$ in case that $\mathfrak A$ is not properly
contained in any other subalgebra of $\mathcal M$ which is
subdiagonal with respect to $\Phi$. Put $\mathfrak A_0 =\{X \in
\mathfrak A: \Phi(X)=0 \}$ and  $\mathfrak A_m =\{X \in \mathcal M:
\Phi(AXB)=\Phi(BXA)=0,~ \forall A \in \mathfrak A,~ B \in \mathfrak
A_0 \}$.
 By   \cite[Theorem 2.2.1]{arv1}, we recall that $\mathfrak A_m $ is a maximal
 subdiagonal algebra of $\mathcal M$ with respect to $\Phi$ containing
 $\mathfrak A$.  $\mathfrak A$ is said to be finite if there is a faithful normal finite trace $\tau$ on $\mathcal M$ such that $\tau\circ\Phi=\tau$.  Finite subdiagonal algebras are  maximal subdiagonal(cf.\cite{exe}). A well known problem  given by Arveson in \cite{arv1} is whether a subdiagonal algebra in a finite von Neumann algebra is automatically finite. It is still open now.

 We  next recall Haagerup's noncommutative $L^p$ spaces associated with a
 general von Neumann algebra $\mathcal M$(cf.\cite{haa,ter}). Let $\varphi$ be a faithful
 normal state on $\mathcal M$ and  let $\{\sigma_t^{\varphi}:t\in\mathbb R\}$
  be the modular automorphism group of $\mathcal M$ associated with $\varphi$ by Tomita-Takesaki theory.
   We consider  the crossed product
 $\mathcal N=\mathcal M\rtimes_{\sigma^{\varphi}} \mathbb R$ of $\mathcal M$
 by $\mathbb R$ with respect to $\sigma^{\varphi}$.
   We denote by $\theta$ the dual action of $\mathbb R$ on $\mathcal
  N$. Then $\{\theta_s:s\in \mathbb R\}$  is an  automorphisms group
  of $\mathcal N$. 
Note that $\mathcal M=\{X\in \mathcal N: \theta_s(X)=X,\forall
s\in\mathbb R\}$.  $\mathcal N$ is a semifinite von Neumann algebra
and there is the normal faithful semifinite trace $\tau$ on
$\mathcal N$ satisfying
$$\tau\circ \theta_s=e^{-s}\tau,  \ \ \  \forall s\in\mathbb R.$$

According to Haagerup \cite{haa,ter}, the noncommutative $L^p$
spaces $L^p(\mathcal M)$ for each $0< p\leq \infty$ is defined as
the set of all $\tau$-measurable  operators $x$   affiliated with
$\mathcal N$ satisfying
$$
\theta_s(x)=e^{-\frac{s}{p}}x,  \ \forall  s\in\mathbb R.$$ There is
a linear bijection  between the predual $\mathcal M_*$ of $\mathcal
M$ and $L^1(\mathcal M)$: $f\to h_f$. If we define
$\mbox{tr}(h_f)=f(I), f\in$$ \mathcal M_*$, then $$
\mbox{tr}(|h_f|)=\mbox{tr}(h_{|f|})=|f|(I)=\|f\|$$ for all
 $f\in \mathcal M_*$ and
$$|\mbox{tr}(x)|\leq\mbox{tr}(|x|)$$ for all  $x\in L^1(\mathcal M)$.  Note that for any $x\in L^p(\mathcal M)$,  $\|x\|_p=(tr(|x|^p))^{\frac1p}$ is the norm of $x$.
As in \cite{haa}, we define the operator $L_A$ and $R_A$ on
$L^p(\mathcal M)$($1\leq p<\infty)$  by $L_Ax=Ax$ and $R_Ax=xA$ for
all  $A\in\mathcal M$ and $x\in L^p(\mathcal M)$.
 Note that
$L^2(\mathcal M)$ is a Hilbert space with the inner product $\langle a,b \rangle
=\mbox{tr}(b^*a)$, $ \forall a,b \in L^2(\mathcal M)$ and
  $A\to L_A$( resp. $A\to R_A$) is a
  faithful normal $*$-representation (resp. $*$-anti-representation) of $\mathcal
  M$ on $L^2(\mathcal M)$. We may identify $\mathcal M$ with
  $L(\mathcal M)=\{L_A:A\in \mathcal M\}$.

   Let $\mathfrak A$ be a maximal subdiagonal algebra with respect to $\Phi$ such that $\varphi\circ\Phi=\varphi$.
It is known that the noncommutative $H^p$ space  $H^p(\mathcal M)$  and $H_0^p(\mathcal M)$ in $L^p(\mathcal M)$
for any $1\leq p<\infty$  is  $H^p=H^p(\mathcal M)=[h_0^{\frac{\theta}{p}}\mathfrak Ah_0^{\frac{1-\theta}p}]_p$ and $H_0^p=H_0^p(\mathcal M)=[h_0^{\frac{\theta}{p}}\mathfrak A_0h_0^{\frac{1-\theta}p}]_p$ for any $\theta\in[0,1]$
   from \cite[Definition
2.6]{jig} and \cite[Proposition 2.1]{jig1}. If $p=\infty$, then we  identify  $H^{\infty}$ as $\mathfrak A$ and $H_0^{\infty}$ as $\mathfrak A_0$.

In this paper, we  consider  type 1 subdiagonal algebras introduced in \cite{jig2}.
 We give    necessary and sufficient conditions for a type 1 subdiagonal algebra to be maximal in $\sigma$-weakly closed subalgebras of a von Neumann algebra. In addition, we devote to Arveson's    finiteness problem of subdiagonal algebras in a finite von Neumann algebra.

\section{Maximality of  type 1 subdiagonal algebras}

 A $\sigma$-weakly closed proper  subalgebra $\mathcal A\subseteq \mathcal M$ is said to be maximal if it can not be  contained  in any  other  $\sigma$-weakly closed proper subalgebras of $\mathcal M$. We consider  conditions for which  a type 1 subdiagonal algebra to be  maximal. We need certain invariant space representations of a maximal subdiagonal algebras in noncommutative $L^p$ spaces. So we firstly recall that the notion of column $L^p$-sum of noncommutative $L^p$-space $L^p(\mathcal M)(1\leq p\leq \infty)$ for a $\sigma$-finite von Neumann algebra $\mathcal M$ studied by Junge and Sherman \cite{js}.  Assume that  $X$ is a closed subspace of $L^p(\mathcal M)$. If
$\{X_i : i\in\Lambda \}$  is a family  of  closed subspaces
of $X$ such that  $X=\vee\{X_i:i\in\Lambda\}$ with the property that $X_j^*X_i=\{0\}$ if $i\not=j$,
then we say that X is the internal column $L^p$-sum  $\oplus_{i\in\Lambda}^{col}X_i$. If $p=\infty$, we  assume
that $X $ and $\{X_i:i\in\Lambda\}$ are $\sigma$-weakly closed, and the  closed linear span is taken with the $\sigma$-weak topology.  For symmetry,  if $X_jX_i^*=\{0\}$ if $i\not=j$ and  $X=\vee\{X_i:i\in\Lambda\}$, then we say that $X$ is the
  internal row $L^p$-sum  $\oplus_{i\in\Lambda}^{row}X_i$. In addition, we also give the following definition  which in fact is used in \cite{ble4,lab}. Let $\mathfrak A $ be a maximal subdiagonal algebra with respect to $\Phi$ and $\mathfrak D$ is the diagonal.
\begin{definition}Let  $\mathcal U=\{U_n:n\geq 1\}$ in $\mathcal M$  be   a family of partial isometries.  If $U_n^*U_m=0$(resp. $U_nU_m^*=0$) for $n\not=m$ and  $U_n^*U_n\in\mathfrak D$(resp. $U_nU_n^*\in\mathfrak D$) for all $n$, then we say that  $\{U_n:n\geq 1\}$ is  column(resp. row) orthogonal.
\end{definition}
 Note that $\mathcal U=\{U_n:n\geq 1\}$ is column orthogonal if and only if $\mathcal U^*=\{U_n^*:n\geq 1\}$ is row orthogonal.

 We recall that a  closed subspace $\mathfrak M$
of $ L^p(\mathcal M)$  is right(resp. left)  invariant if  $\mathfrak M
\mathfrak A\subseteq \mathfrak M$(resp. $ \mathfrak A\mathfrak M
\subseteq \mathfrak M$).  If it is both left and  right invariant, then we say that it is two-side invariant.   Following \cite{ble4,lab}, we define the right wandering subspace of
$\mathfrak M$  to be the space $K = \mathfrak M\ominus [\mathfrak M\mathfrak A_0]_2$ when $p=2$, where $[S]_p$ is the closed linear span of a subset $S$ in $L^p(\mathcal M)$. We say that $\mathfrak M$ is of  type 1 if $K$ generates $\mathfrak M$ as an
$\mathfrak A$-module (that is, $\mathfrak M = [K\mathfrak A]_2$). We   say that $\mathfrak M$ is  of type 2 if $K = \{0\}$.
 Note that every   right invariant subspace $\mathfrak M$ is an $L^2$-column sum $\mathfrak M=\mathfrak N_1\oplus^{col}\mathfrak N_2$, where     $\mathfrak N_i$  is of  type $i$ for $i=1,2$ from \cite[Theorem 2.1]{ble4} and
\cite[Theorem 2.3]{lab}.  In particular,   if $\mathfrak M$ is of   type 1, then $\mathfrak M$ is of the Beurling type, that is, there exists a family  of column orthogonal partial isometries $\{U_n:n\geq 1\}$  such that $\mathfrak M=\oplus_{n}^{col}U_nH^2 $.  We refer    to {\cite{ble4,lab} for  more details. Symmetrically, a type 1 left invariant subspace $\mathfrak M$  may be represented as
 $\mathfrak M=\oplus_{n\geq 1}^{row}H^2V_n$ by a family of row orthogonal partial isometries $\{V_n:n\geq1\}$ and this fact will be used frequently.

  By  \cite[Definition 2.1]{jig2}, $\mathfrak A$ is said to be type 1 if every right invariant subspace of $\mathfrak A$ in $ H^2$ is of  type 1. Then   there exists a column orthogonal family of  partial isometries $\{U_n:n\geq1\}$ in $\mathcal M$  such that   \begin{equation}
 H_0^2 = \oplus_{n\geq 1}^{col} U_n  H^2.
 \end{equation}
 To consider the maximality of $\mathfrak A$, we firstly  consider invariant subspaces  in non commutative $L^1(\mathcal M)$ and 
 $\mathcal M$.  

 For a family of column orthogonal partial isometries
  $\mathcal W=\{W_n:n\geq 1\}$, we define  right invariant subspaces
  $\mathfrak M_{\mathcal W}^1=\oplus_{n\geq 1}^{col}W_n H^1$ in $ L^1(\mathcal M)$ and $\mathfrak M_{\mathcal W}^{\infty}=\oplus_{n\geq 1}^{col}W_n H^{\infty}$ in $\mathcal M$ respectively.
 
\begin{theorem}
Let   $\mathfrak A$ be a type 1 subdiagonal algebra  with respect to $\Phi$.   
   
   $(1)$ If  $\mathfrak M\subseteq L^1(\mathcal M)$   is a closed  right invariant subspace, then there exist a family of column orthogonal partial isometries $\mathcal W$  and a projection $E$ in $\mathcal M$ and such that
   $\mathfrak M=\mathfrak M_{\mathcal W}^1\oplus^{col}EL^1(\mathcal M)$.

   $(2)$ If  $\mathfrak M\subseteq \mathcal M$   is a $\sigma$-weakly closed  right invariant subspace, then there exist a family of column orthogonal partial isometries $\mathcal W$  and a projection $E$ in $\mathcal M$ and such that
   $\mathfrak M=\mathfrak M_{\mathcal W}^{\infty}\oplus^{col}E\mathcal M$.

\end{theorem}
\begin{proof}

$(1)$ 
 Put $\mathfrak M_2=\cap_{n\geq 1}[\mathfrak M\mathfrak A_0^n]_1$, where  $\mathfrak A_0^n$ is  the  $\sigma$-weakly closed ideal of $\mathfrak A$  generated by $\{a_1a_2\cdots a_n:a_j\in\mathfrak A_0\}$. Then $\mathfrak M_2\subseteq \mathfrak M$ is a right invariant subspace of  $\mathfrak A$.  We show that $\mathfrak M_2=EL^1(\mathcal M)$ for a projection $E\in\mathcal M$.
 By \cite[Theorem 2.7]{jig2},
$\mathfrak A_0=\vee\{\mathfrak DU_{i_1}U_{i_2}\cdots U_{i_n}:i_j\geq 1\}$, where $\{U_i:i\geq 1\}$ is a family of column orthogonal partial isometries as in $(2.1)$.
 Note that $\mathfrak M$ is  $\mathfrak D$ right  invariant  and   $[\mathfrak M_2\mathfrak A_0]_1=\mathfrak M_2$.  Since
 $\vee\{\mathfrak M_2U_n:n\geq1\}\supseteq \vee\{\mathfrak M_2U_{i_1}U_{i_2}\cdots U_{i_n}:i_k\geq 1,n\geq1\}$, $\mathfrak M_2=\vee\{\mathfrak M_2 U_n:n\geq1\}$.  For any $m,n \geq 1$ and $x\in \mathfrak M_2$, we have $R_{U_m^*}R_{U_n}x=xU_nU_m^*\in \mathfrak M_2$ since $U_nU_m^*\in\mathfrak D$ from \cite[Proposition 2.3]{jig2}.   Then $\mathfrak M$ is right $\mathcal M$ invariant.  By \cite[Chapter III, Theorem 2.7]{tak},
 there is a projection $E\in\mathcal M$ such that $\mathfrak M_2=EL^1(\mathcal M)$.
   Put $\mathfrak M_1=(I-E)\mathfrak M\subseteq \mathfrak M$. This mean that $\mathfrak M_1$ is  closed and right invariant
 such that $\mathfrak M=\mathfrak M_1\oplus^{col}\mathfrak M_2$. Note that  $\cap_{n\geq 1}[\mathfrak M_1\mathfrak A_0^n]_1\subseteq \cap_{n\geq 1}[\mathfrak M\mathfrak A_0^n]_1\subseteq\mathfrak M_1\cap\mathfrak M_2=\{0\}$.

 Put 
 \begin{align*}\  \mathcal P =  &\{\mathcal W=\{W_n\}_{n\geq 1}\subseteq\mathcal M: 
  \mbox{column orthogonal  }  \\
  &  \mbox{ partial isometries such that }  \mathfrak M_{\mathcal W}^2 H^2\subseteq \mathfrak M_1  \},\end{align*}  where $\mathfrak M_{\mathcal W}^2 H^2$ is the closed right invariant subspace generated by $\{xy: x\in \mathfrak M_{\mathcal W}^2,y\in  H^2\}$.  We define a partial order in $\mathcal P$ by $\mathcal W\leq \mathcal V$ if $\mathfrak M_{\mathcal W}^2\subseteq \mathfrak M_{\mathcal V}^2$ for any $\mathcal W, \mathcal V\in \mathcal P$. Let $\{\mathcal W_{\lambda}:\lambda \in\Lambda\}\subseteq \mathcal P$ be a totally ordered family in $\mathcal P$. Put $\mathfrak N=\vee\{\mathfrak M_{\mathcal W_{\lambda}}^2: \lambda\in\Lambda\}$. Note that $\mathfrak N\subseteq  L^2(\mathcal M)$ is a right invariant subspace in $L^2(\mathcal M)$. Then  $\mathfrak N=\mathfrak M_{\mathcal W}^2\oplus^{col} \mathfrak N_2$ for a family of column orthogonal partial isometries $\mathcal W$ and a right invariant subspace $\mathfrak N_2$ of type 2 in $L^2(\mathcal M)$.  Then $[\mathfrak N_2 H^2]_1$ is also a  right invariant subspace in $ L^1(\mathcal M)$ such that $[\mathfrak N_2 H^2\mathfrak A_0]_1= [\mathfrak N_2 H^2]_1$.
 This implies that $[\mathfrak N_2 H^2]_1=\cap_{n\geq 1} [(\mathfrak N_2 H^2)\mathfrak A_0^n]_1
  \subseteq \cap_{n\geq 1}[\mathfrak M_1\mathfrak A_0^n]_1=\{0\}$.   Hence $\mathfrak N_2=\{0\}$.
 Since $\mathfrak M_{\mathcal W_{\lambda}}^2 H^2\subseteq \mathfrak M_1$, $\mathfrak N H^2\subseteq \mathfrak M_1$. Thus
  $\mathcal W\in\mathcal P$ with $\mathcal W_{\lambda}\leq \mathcal W$ for any $\lambda\in \Lambda$. That is, $\mathcal W$ is an upper bound of $\{\mathcal W_{\lambda}:\lambda\in\Lambda\}$. By Zorn's lemma, there exists a maximal element $\mathcal W\in \mathcal P$. We show  that $\mathfrak M_1=\mathfrak M_{\mathcal W}^1$.

  Otherwise, assume that  there is an $h\in \mathfrak M_1$ such that $h\notin\mathfrak M_{\mathcal W}^1$. Then by
  considering the polar decomposition of $h$,  there are  $h_1\in  L^2(\mathcal M)$ and an outer element $h_2\in  H^2$ such that  $h=h_1h_2$ by \cite[Theorem 3.1]{jig2}. This implies that $\mathfrak M_1\supseteq [h\mathfrak A]_1=[h_1\mathfrak A]_2 H^2$.
    Note that $[h_1\mathfrak A]_2\subseteq  L^2(\mathcal M)$ is a right invariant subspace. By \cite[Lemma 3.3]{jig2},
      $[h_1\mathfrak A]_2=V H^2\oplus^{col} N_2$ for a partial isometry $V\in\mathcal M$ such that $V^*V\in\mathfrak D$
      and a  right invariant subspace $N_2$ of type 2.  Note that  $[h_1\mathfrak A]_2 H^2\subseteq \mathfrak M_1$.
       Then $N_2=0$.

     Since $h=h_1h_2\notin \mathfrak M_{\mathcal W}^1$, $V H^2\nsubseteq\mathfrak M_{\mathcal W}^2$. Now
$\tilde{\mathfrak N}=\vee\{\mathfrak M_{\mathcal W}^2,V H^2\}$ is also a right invariant subspace in $ L^2(\mathcal M)$ with $\tilde{\mathfrak N} H^2\subseteq \mathfrak M_1$.  In this case, we have $\tilde{\mathfrak N}=\mathfrak M_{\mathcal V}^2$ for a family of  column orthogonal partial isometries $\mathcal V\in\mathcal P$. It is trivial  that $\mathfrak M_{\mathcal W}^2\subsetneqq\mathfrak M_{\mathcal V}^2$ by a similar treatment. This is a contradiction. Thus
$\mathfrak M_1=\mathfrak M_{\mathcal W}^2H^2=\mathfrak M_{\mathcal W}^1$.

 $(2)$ Let $p=\infty$.   Note that  $[\mathfrak M h_0^{\frac12}]_2=\mathfrak M H^2\subseteq L^2(\mathcal M)$ is a right invariant invariant subspace of $\mathfrak A$ in $L^2(\mathcal M)$  since $[\mathfrak M\mathfrak A]_{\infty}=\mathfrak M$.
 Thus there are  a family of column orthogonal partial isometries $\mathcal W=\{W_n:n\geq 1\}$ and a projection $E$ in $\mathcal M$  such that
 $\mathfrak MH^2=M_{\mathcal W}^2\oplus^{col} EL^2(\mathcal M)=\left(\mathfrak M_{\mathcal W}^{\infty}\oplus^{col}E\mathcal M\right)H^2$.

 We next show that $\mathfrak M=\mathfrak M_{\mathcal W}^{\infty}\oplus^{col}E\mathcal M$. For any $x\in\mathfrak M$, it is elementary
 that $xh_0^{\frac12}=\oplus^{col}_{n\geq 1}W_nW_n^*xh_0^{\frac12}+Exh_0^{\frac12}$ with $W_n^*xh_0^{\frac12}\in H^2$. Then $W_n^*x\in \mathfrak A$ and
 $x=\oplus^{col}_{n\geq 1}W_nW_n^*x\oplus Ex\in \mathfrak M_{\mathcal W}^{\infty}\oplus E\mathcal M$. Therefore $\mathfrak M\subseteq \mathfrak M_{\mathcal W}^{\infty}\oplus E\mathcal M$.

 For any closed subspace $K\subseteq L^p(\mathcal M)$, we put $K^{\bot}=\{y\in L^q(\mathcal M):tr(xy)=0,\forall x\in K\}$, where $p$ and $q$ are conjugate exponents, that is, $\frac1p+\frac1q=1$.
 It is known that $K^{\bot}$ is  left invariant when $K$ is right invariant. We easily have $\mathfrak M^{\bot}\subseteq L^1(\mathcal M)$ is a left invariant subspace. As just proved,  by symmetry, we have that
 $\mathfrak M^{\bot}=\oplus_{n\geq 1}^{row} H^1V_n\oplus ^{row} L^1(\mathcal M)F$  for a family of row orthogonal partial family $\{V_n:n\geq 1\}$ and a projection $F\in\mathcal M$.  It is elementary that $\mathfrak M^{\bot}=H^2\left(\oplus_{n\geq 1}^{row} H^2V_n\oplus ^{row} L^2(\mathcal M)F\right)=H^2\mathfrak N$, where   $\mathfrak N=\oplus_{n\geq 1}^{row} H^2V_n\oplus ^{row} L^2(\mathcal M)F\subseteq L^2(\mathcal M)$ is left invariant.

  We claim that $(\mathfrak MH^2)^{\bot}=\mathfrak N$. It is clear that $(\mathfrak MH^2)^{\bot}\supseteq \mathfrak N$. Take any $y\in(\mathfrak MH^2)^{\bot}$. Then $H^2y\subseteq \mathfrak M^{\bot}=\oplus_{n\geq 1}^{row} H^1V_n\oplus ^{row} L^1(\mathcal M)F$.
 By a similar treatment as above, we have
 $h_0^{\frac12}y=\oplus_{n\geq 1}^{row}h_0^{\frac12}yV_n^*V_n \oplus^{row}h_0^{\frac1r}yF$ and thus $y=\oplus_{n\geq 1}^{row}yV_n^*V_n \oplus^{row}yF\in \mathfrak N$.
Thus,  $\mathfrak M^{\bot}=H^2\mathfrak N=H^2(\mathfrak MH^2)^{\bot}$.
 On the other hand,   replacing $\mathfrak M$ by $\mathfrak M_{\mathcal W}^{\infty}+E\mathcal M$, we have   
    $ \left(\mathfrak M_{\mathcal M}^{\infty}+E\mathcal M\right)^{\bot}=H^2\left((\mathfrak M_{\mathcal W}^{\infty}+E\mathcal M)H^2\right)^{\bot}=
     H^2(\mathfrak MH^2)^{\bot}=\mathfrak M^{\bot}$.
 It follows that  $\mathfrak M=\mathfrak M_{\mathcal W}^{\infty}+E\mathcal M$.
\end{proof}

We  now consider the maximality of a type 1 subdiagonal algebra  $\mathfrak A$ as a $\sigma$-weakly closed subalgebra in a von Neumann algebra $\mathcal M$. We next assume that $\mathcal B$ is  a $\sigma$-weakly closed proper subalgebra of $\mathcal M$ such that $ \mathfrak A\subseteq \mathcal B \varsubsetneq\mathcal M$.

\begin{proposition}$\mathcal B=\{T\in\mathcal M:[h_0^{\frac12}\mathcal B]_2T\subseteq [ h_0^{\frac12}\mathcal B]_2\}=\{T\in\mathcal M:T[\mathcal Bh_0^{\frac12}]_2\subseteq [ \mathcal Bh_0^{\frac12}]_2\}$.
\end{proposition}

\begin{proof}
As in the proof of Theorem 2.2, we put $\mathcal B^{\bot}=\{h\in  L^1(\mathcal M): \mbox{tr}(hB)=0,\forall B\in\mathcal B\}$ is the pre-annihilator of $\mathcal B$. It is known that $(\mathcal B^{\bot})^{\bot}=\mathcal B$ since $\mathcal B$ is $\sigma$-weakly closed. Note that $\mathcal B^{\bot}\subseteq  H^1_0$ is a two-side $\mathcal B\supseteq \mathfrak A$ invariant subspace. By Theorem  2.2,     $\mathcal B^{\bot}=\mathfrak M_{\mathcal W}^1=\mathfrak M_{\mathcal W}^2H^2$ for a family of column orthogonal partial isometries $\mathcal W$ since $[\mathcal B^{\bot}\mathfrak A_0^n]_1\subseteq [H_0^1\mathfrak A_0^n]_1=\{0\}$. Take a $T\in\mathcal M$ such that $[ h_0^{\frac12}\mathcal B]_2T\subseteq [ h_0^{\frac12}\mathcal B]_2$. For any $h_1\in\mathfrak M_{\mathcal W}^2$ and $h_2\in H^2\subseteq [h_0^{\frac12}\mathcal B]_2$, $h_2T\in[h_0^{\frac12}\mathcal B]_2$ and then $h_2T=\lim\limits_{m\to\infty}h_0^{\frac12}B_m$ for a sequence $\{B_m:m\geq 1\}\subseteq \mathcal B$. Now
$\mbox{tr}(Th_1h_2)=\mbox{tr}(h_2 Th_1)=\lim\limits_{m\to \infty}\mbox{tr}(h_0^{\frac12}B_mh_1)=\lim\limits_{m\to \infty}\mbox{tr}(B_mh_1h_0^{\frac12})=0$.  It follows that $T\in \left(\mathcal B^{\bot}\right)^{\bot}=\mathcal B$. By symmetry, we also have the second equality.
\end{proof}

\begin{corollary} Put $Q=\vee \{E\in\mathcal M:E\mathcal M\subseteq \mathcal B\}$. Then $(I-Q)\mathcal BQ=0$. \end{corollary}

\begin{proof}

 Note that $Q\mathcal M\subseteq \mathcal B$.  Put $\mathcal N=[(I-Q)\mathcal BQ\mathcal Mh_0^{\frac12}]_2\subseteq [\mathcal Bh_0^{\frac12}]_2$. Then $\mathcal N$ is right reducing and $\mathcal N=FL^2(\mathcal M)$ for a projection $F\in\mathcal M$.
 It follows that $F\mathcal M\subseteq \mathcal B$  by Proposition 2.3 since $F\mathcal M[\mathcal Bh_0^{\frac12}]_2\subseteq \mathcal N\subseteq [\mathcal Bh_0^{\frac12}]_2$ and thus $F\leq Q$. However, $(I-Q)F=F$.  Hence $F=0$ and $(I-Q)\mathcal BQ=0$.
 \end{proof}

   We now consider main result in this section.  If $\mathfrak M\subseteq L^2(\mathcal M)$ is a type 1 right invariant subspace of $\mathfrak A$, then the projection  on the wandering subspace $K=\mathfrak M\ominus [\mathfrak M\mathfrak A_0]_2$ is in the commutant $(R(\mathfrak D))^{\prime}$ of  the von Neumann algebra $R(\mathfrak D)=\{R_D:D\in\mathfrak D\}$. For any two projections $E$ and $F$ in a von Neumann algebra $\mathcal M$, $E\preccurlyeq F$ means that there is a partial isometry $V\in\mathcal M$ such that $V^*V=E$ and $VV^*\leq F$. If $E\preccurlyeq F$ and $F\preccurlyeq E$, then  $E\sim F$. The following proposition holds for a maximal subdiagonal algebra.

\begin{proposition}
 Let $\mathfrak M_i(i=1,2)$ be two type 1 right invariant subspaces of $\mathfrak A$ with wandering subspaces $ K_i(i=1,2)$. if  $p_i$ are  the projections on $ K_i(i=1,2) $ and  $p_1\preccurlyeq p_2$ in $(R(\mathfrak D))^{\prime}$, then there is a partial isometry $W\in\mathcal M$ such that $\mathfrak M_2=W\mathfrak M_1\oplus^{col} (I-WW^*)\mathfrak M_2$  and
  $\mathfrak M_1=W^*\mathfrak M_2$.\end{proposition}
  \begin{proof}
  Let $w\in (R(\mathfrak D))^{\prime}$ be a partial isometry such that $w^*w=p_1$ and $ww^*\leq p_2$. Note that
  $[ K_i\mathcal M]_2=[ K_i\mathfrak A_0^*]_2\oplus [K_i\mathfrak D]_2\oplus[ K_i\mathfrak A_0]_2$ are right reducing subspaces. For any $x\in K_1$, we have $wx\in K_2$ and $x^*x,(wx)^*(wx) \in L^1(\mathfrak D)$ by \cite[Theoren 2.3]{lab}. Note that  $R_Dx=xD\in\mathfrak M_1$ and $wR_Dx=R_D(wx)=(wx)D$ and therefore  $\|wR_Dx\|_2=\|R_Dx\|_2=\|xD\|_2$.
   Take any $A\in\mathcal M$.  Recall that $\Phi_1$ is the contraction from $L^1(\mathcal M)$  onto  $L^1(\mathfrak D)$. Since $\mbox{tr}\circ\Phi_1=\mbox{tr}$ from \cite[Proposition 2.1]{jig},
  we have that \begin{align*}
  \|R_A(wx)\|_2^2&=\|(wx)A\|_2^2=\mbox{tr}(A^*(wx)^*(wx)A)=\mbox{tr}(|wx|^2AA^*)\\
  &=\mbox{tr}(|wx|^2\Phi(AA^*))=\mbox{tr}((\Phi(AA^*)^{\frac12}(wx)^*(wx)(\Phi(AA^*))^{\frac12})\\
  &=\|(wx)(\Phi(AA^*)^{\frac12}\|_2^2=\|(R_{\Phi(AA^*))^{\frac12}}(wx)\|_2^2\\
  &=\|wR_{(\Phi(AA^*))^{\frac12}}x\|_2^2=\|R_{(\Phi(AA^*))^{\frac12}}x\|_2^2\\
  &=\mbox{tr}((\Phi(AA^*)^{\frac12}x^*x(\Phi(AA^*))^{\frac12})=\mbox{tr}(x^*x\Phi(AA^*))\\
  &=\mbox{tr}(x^*xAA^*)=\|xA\|_2^2
  .\end{align*}
  We define $W(xA)=(wx)A$ for any $x\in  K_1$ and $A\in\mathcal M$. Then $W$ is  an well defined isometry  on $[ K_1\mathcal M]_2$ with range $[w( K_1)\mathcal M]_2$.  Put $Wx=0$ for any $x\in [ K_1\mathcal M]_2^{\bot}$. Then $W $ is a partial isometry. Note that $WR_B(xA)= W(xAB)=(wx)AB=R_B(R_A(wx))=R_BW(xA)$ for all $A, B\in\mathcal M$ and $x\in  K_1$. On the other hand, if $x\in [ K_1\mathcal M]_2^{\bot}$, then $WR_Bx=W(xB)=0=R_B(Wx)$ since $ [ K_1\mathcal M]_2$ is right reducing. This implies that $W\in R(\mathcal M)^{\prime}=L(\mathcal M)$.   Note that $W( K_1)=w( K_1)\subseteq  K_2$ is a right $\mathfrak D$ module, so is $W_2\ominus W(K_1)$. It is elementary  that  $W\mathfrak M_1=[W( K_1)\mathfrak A]_2\subseteq \mathfrak M_2$  and $[\left(K_2\ominus W(K_1)\right)\mathfrak A]_2 $ are    right invariant subspaces  of $\mathfrak A$ of type 1 such that $\mathfrak M_2=W\mathfrak M_1\oplus^{col}
  [\left(K_2\ominus W(K_1)\right)\mathfrak A]_2  $. It is trivial that $\mathfrak M_1=[ K_1\mathfrak A]_2=[W^*W( K_1)\mathfrak A]_2=W^*[(W( K_1)\mathfrak A]_2$. It is known that $W^*(yA)=0$ for all $y\in K_2\ominus w
  K_1$ and $A\in\mathcal M$. Therefore $\mathfrak M_1=W^*\mathfrak M_2$ and $[\left(K_2\ominus W(K_1)\right)\mathfrak A]_2 =(I-WW^*)\mathfrak M_2$.
  \end{proof}

  \begin{theorem} Let $\mathfrak A\subsetneq \mathcal M$ be a type 1 subdiagonal algebra with respect to $\Phi$. Then $\mathfrak A$ is a maximal $\sigma$-weakly closed  subalgebra of  $\mathcal M$ if and only if one of  following  assertions holds.

  $(1)$ There is a  projection $E\in\mathcal M$ which is not in the center $Z(\mathcal M)$ such that $\mathfrak A=E\mathcal M+(I-E)\mathcal M(I-E)=\{A\in\mathcal M:(I-E)AE=0\}$.

  $(2)$ There is a projection $E\in Z(\mathcal M)\cap \mathfrak D$ such that $E\mathcal M =E\mathfrak A$ and $(I-E)\mathfrak D$ is a factor.
  
  In particular, if $\mathfrak D$ is a factor, then $\mathfrak A$ is maximal.\end{theorem}
\begin{proof}
If  assertion $(1)$ holds, then it is trivial that $\mathfrak A$ is maximal.

We assume that assertion $(2)$ holds. Then $\mathfrak A=E\mathcal M+E\mathfrak A$. In this case, $E\mathfrak A$ is a type 1 subdiagonal algebra  of $E\mathcal M$ with respect to $\Phi_E$, where $\Phi_E(EA)=E\Phi(A)$ for all $A\in\mathcal M$. It is sufficient to prove that $E\mathfrak A$ is a maximal  subalgebra in $E\mathcal M$. Without loss of generality, we may assume that $\mathfrak D$ itself  is a factor.

Let $\mathcal B$ be a $\sigma$-weakly closed subalgebra of $\mathcal M$ such that $\mathfrak A\subseteq \mathcal B\subsetneqq \mathcal M$. We show that $\mathcal B=\mathfrak A$. Since $\mathcal B$ is two-sided $\mathfrak A$ invariant, by Theorem 2.2, there are a family of column orthogonal partial isometries $\mathcal W=\{W_m:m\geq 1\}$ and a projection $P$ in  $\mathcal M$ such that
$\mathcal B=\oplus_{m\geq1}^{col}W_m\mathfrak A\oplus^{col}P \mathcal M$. Note that $P<I$ by Proposition 2.3 since $\mathcal B$ is a proper subalgebra of $\mathcal M$.
  Since $PL^2(\mathcal M)=\cap_{n\geq 1}[\mathcal Bh_0^{\frac12}\mathfrak A_0^n]_{2}$ and $[\mathcal Bh_0^{\frac12}\mathfrak A_0^n]_{2}$ is also left $\mathcal B$ invariant for any $n$, $PL^2(\mathcal M)$ is also left $\mathcal B$ invariant. This implies that $PBP=BP$ for any $B\in\mathcal B$.
    In particular, $PD=DP$ for all $D\in\mathcal B\cap \mathcal B^*$.  Note that $W_m\in \mathcal B\cap \mathcal B^*$ and $PW_m=0$ for any $m$.   Thus $PW_m=W_mP=0$ and  $W_m^*W_m\leq I-P$ for any $m$.
    On the other hand, $P\in \left(\mathcal B\cap \mathcal B^*\right)^{\prime}\subseteq \mathfrak D^{\prime}$. We have that $\Phi(P)\in Z(\mathfrak D)=\mathbb C I$. This implies that $0\leq \Phi(P)<1$ is a scalar.
   However,
   $W_m^*W_m=\Phi(W_m^*W_m)\leq I-\Phi(P)$. It follows that
 $\Phi(P)=0$ since $W_m^*W_m$ is a nonzero projection for any $m$. Thus $P=0$ and $[\mathcal Bh_0^{\frac12}]_2$ is a type 1 right invariant subspace. Let $p$ and $q$ be the projections on $L^2(\mathfrak D)$ and
 $[\mathcal Bh_0^{\frac12}]_2\ominus [\mathcal Bh_0^{\frac12}\mathfrak A_0]_2$, the wandering subspaces of $H^2$ and $[\mathcal Bh_0^{\frac12}]_2$ respectively. Then $p$ and $q$  are in the factor $\left(R(\mathfrak D)\right)^{\prime}$. Thus $p\preccurlyeq q$ or $q\preccurlyeq p$ in $\left(R(\mathfrak D)\right)^{\prime}$.

   Case 1.  $p\preccurlyeq q$.  Then $H^2=W^*[\mathcal Bh_0^{\frac12}]_2$ for a partial isometry $W\in\mathcal M$ such that $WW^*H^2=[w(L^2(\mathfrak D)\mathfrak A]_2=WH^2\subseteq [\mathcal Bh_0^{\frac12}]_2$  and $[\mathcal Bh_0^{\frac12}]_2=WH^2\oplus^{col}(I-WW^*)[\mathcal Bh_0^{\frac12}]_2$ by Proposition 2.5. It is trivial that $W$ is an isometry since $h_0^{\frac12}$ is in the initial space of $W$. On the other hand,  $W^*\mathcal B\subseteq \mathfrak A$ and $W\mathfrak A\subseteq \mathcal B$ by Proposition 2.3. Thus  $\mathfrak A=W^*\mathcal B=W^*\mathcal B\mathcal B=\mathfrak A\mathcal B=\mathcal B$.

   Case 2. $q\preccurlyeq p$. Then $[\mathcal Bh_0^{\frac12}]_2=W^*H^2$ for a partial isometry $W\in\mathcal M$. Thus $W$ is  a co-isometry.  However $W^*W[\mathcal Bh_0^{\frac12}]_2=W^*H^2 =[\mathcal Bh_0^{\frac12}]$. Then  $W $ is unitary and again $\mathcal B=\mathfrak A$.

   Conversely, we assume that $\mathfrak A$ is maximal in $\mathcal M$.
   We claim that  for any $E\in Z(\mathfrak D)$, either $E\mathcal M\subseteq \mathfrak A$ or $(I-E)\mathcal M\subseteq \mathcal M$.

  Let $0<E<I$. If $EU_n^*\in\mathfrak A$ for all $n $, then $E\mathfrak A^*\subseteq \mathfrak A$ and thus $E\mathcal M\subseteq \mathfrak A$. Assume that $EU^*_n\notin \mathfrak A$ for some $n$. Then the $\sigma$-weakly closed subalgebra generated by $\{EU^*_n:n\geq 1\}$ and $\mathfrak A$ is $\mathcal M$. Thus  $E\mathfrak A^*+\mathfrak A$ is $\sigma$-weakly dense in $\mathcal M$. In particular,  $(I-E)\mathcal M=(I-E)\mathfrak A\subseteq \mathfrak A$.

  Put $Q=\vee\{E\in \mathcal M: E\mathcal M\subseteq \mathfrak A\}$. By \cite[Lemma 3.1]{jig2},  $Q\mathcal M\subseteq \mathfrak A$ and $(I-Q)\mathfrak AQ=0$. In particular, $Q\in Z(\mathfrak D)$.

  Case 1.    $Q\notin Z(\mathcal M)$. Then $Q\mathcal M(I-Q)\not=0$.  In this case we have that $\mathfrak A=Q\mathcal M+(I-Q)\mathcal M(I-Q)$ since $\mathfrak A$ is maximal. Thus (1) holds.

   Case 2. $Q\in Z(\mathcal M)$. Then  $(I-Q)\mathfrak A$ is a maximal subalgebra of $(I-Q)\mathcal M$. In this case, $(I-Q)\mathfrak D$ is a factor. Otherwise, as the above claim, there is a central projection $E\in Z((I-Q)\mathfrak D)$ such that $E(I-Q)\not=0$ and $E(I-Q)\mathcal M\subseteq (I-Q)\mathfrak A$. This is a contradiction.  Therefore (2) holds.
\end{proof}

\section{Finiteness of type 1 subdiagonal algebras}
 We recall that a subdiagonal algebra  with respect to $\Phi$ in $\mathcal M$ is finite if there is a faithful normal finite trace $\tau$ on $\mathcal M$ such that $\tau\circ\Phi=\tau$. A longstanding open problem given in \cite{arv1} by Arveson is whether a subdiagonal algebra in a finite von Neumann algebra is automatically finite. It is still open from now on.  We answer this problem for type 1 case. We recall that all results  for type 1 subdiagonal algebras  in general von Neumann algebras hold for a finite von Neumann algebra if we replace Haagerup's noncommutative $L^p(\mathcal M)$ by the noncommutative $L^p$ space $L^p(\mathcal M,\tau)$ for a von Neumann algebra $\mathcal M$ with a a faithful normal finite trace $\tau$.
 \begin{theorem} Let $\mathcal M$ be a finite von Neumann algebra with a faithful normal finite trace $\tau$.  If $\mathfrak A$ is a type 1 subdiagonal algebra with respect to $\Phi$ in $\mathcal M$, then $\mathfrak A$ is finite, that is, there exists a faithful normal finite trace $\rho$ on $\mathcal M$ such that $\rho\circ \Phi=\rho$.\end{theorem}
\begin{proof}
 Let $L^2(\mathcal M,\tau)$ be the noncommutative $L^2$ space associated with $\tau$. Then $\mathcal M\subseteq L^2(\mathcal M,\tau)$.
 We choose a faithful normal state $\varphi$ on $\mathcal M$ such that $\varphi\circ\Phi=\varphi$. If  $h_0$  is the image of $\varphi$ in $ L^1(\mathcal M,\tau)$ , that is $\varphi(A)=\tau(Ah_0)$ for all $A\in\mathcal M$, then  nonocommutative $H^2$ and $H_0^2$ are defined similarly.  Note that $\mathfrak A$ is of type 1. $H_0^2=\oplus^{col}_nU_nH^2$ for a family of column orthogonal  partial isometries $\{U_n:n\geq 1\}$ in $\mathcal M$ such that $U_n^*U_m=0$ and  $U_n^*U_n\in\mathfrak D$ as in formula $(2.1)$. Therefore
 \begin{equation}I=x^*+d+x=(\oplus_{n\geq1} U_nx_n)^*+d+\oplus_{n\geq1} U_nx_n,
  \end{equation}
 where  $d\in L^2(\mathfrak D)_+$,    $\{x_n:n\geq 1\}\subset H^2$ and $x=\oplus_{n\geq 1} U_nx_n$. It follows that 
 $$D=D(\oplus_{n\geq1} U_nx_n)^*+Dd+D(\oplus_{n\geq 1} U_nx_n)=(\oplus_{n\geq1} U_nx_n)^*D+dD+(\oplus_{n\geq 1} U_nx_n)D$$
  for any $D\in\mathfrak D$. Since  $D(\oplus_{n\geq 1} U_nx_n), (\oplus_{n\geq 1} U_nx_n)D\in H_0^2$, we have $Dd=dD$ for all $D\in \mathfrak D$ by $(3.1)$.

  Again by $(3.1)$, we have  for any $m\geq 1$, 
  \begin{align*}U_m &=U_m(\oplus_{n\geq1} U_nx_n)^*+U_md+U_m(\oplus_{n\geq1} U_nx_n)\\
  &=(\oplus_{n\geq1} U_nx_n)^*U_m+dU_m+(\oplus_{n\geq1} U_nx_n)U_m.\end{align*}  It follows that
 $$U_md- dU_m=U_m(\oplus_{n\geq 1} U_nx_n)^* +U_m(\oplus_{n\geq1} U_nx_n)-(\oplus_{n\geq 1} U_nx_n)^*U_m-(\oplus_{n\geq1} U_nx_n)U_m .$$ We recall that $\tau(y^*x)=\langle x,y\rangle$ denote the inner product  for any $x,y\in L^2(\mathcal M,\tau)$.
 It is elementary that $\langle U_md, U_mU_nx_n\rangle =\langle d,U_m^*U_mU_nx_n\rangle=0$ since $U_m^*U_m\in\mathfrak D$. Similarly, we have $\langle U_md, U_nx_nU_m\rangle =\langle U_n^*U_md,  x_nU_m\rangle=0$. We also note that
   $$\langle U_md, U_m(x_nU_m)^*\rangle=\langle U_m^*U_md, (x_nU_n)^*\rangle=0$$ and $$\langle U_md, (U_nx_n)^*U_m\rangle=\langle U_md, x_n^*U_n^* U_m\rangle=\langle d,U_m^* x_n^*U_n^* U_m\rangle= 0.$$ This implies that
   $$\langle U_md, U_m(\oplus_{n\geq1} U_nx_n)^* +U_m(\oplus_{n\geq1} U_nx_n)-(\oplus_{n\geq1} U_nx_n)^*U_m-(\oplus_{n\geq1} U_nx_n)U_m\rangle=0.$$  By the same way, we have
   $$\langle  d U_m, U_m(\oplus_{n\geq1} U_nx_n)^* +U_m(\oplus_{n\geq1} U_nx_n)-(\oplus_{n\geq1} U_nx_n)^*U_m-(\oplus_{n\geq1} U_nx_n)U_m\rangle=0.$$ This means that  $\langle U_md-dU_m,U_md-dU_m\rangle=0$ and thus $U_md-dU_m=0$. By \cite[Theorem 2.7]{jig2}, we know that $\mathcal M$ is the  von Neumann  algebra generated  by $\mathfrak D$ and $\{U_m:m\geq1\}$. It now follows that  $Ad=dA$ and therefore
   \begin{equation} Ad^2=d^2A
   \end{equation}
    for all $A\in\mathcal M$.
    Now let $E$ be the support projection of $d^2\in L^1(\mathfrak D)$(cf.\cite[Chapter III, Definition 3.7]{tak}). Then $E\in\mathfrak D$ and $(I-E)d=d(I-E)=0$. However,
    \begin{align*}
    I-E &=(I-E)x^*(I-E)+(I-E)d(I-E)+(I-E)x(I-E)\\
    &=(I-E)x^*(I-E) +(I-E)x(I-E)\end{align*} by $(3.1)$. Note that $I-E\geq 0$ and $(I-E)x^*(I-E) +(I-E)x(I-E)\in (H_0^2)^*+H_0^2$. It follows that $I-E=0$. Thus $d\in L^2(\mathfrak D)$ is  both left and right separating.
   Put $\rho(A)=\langle Ad,d\rangle=\tau(Ad^2)$, $\forall A\in\mathcal M$.  Then  $\rho$ is a faithful normal state on $\mathcal M$. It is trivial that $\rho\circ \Phi=\rho$ since $d^2\in L^1(\mathfrak D)$.  Now for any $A,B \in\mathcal M$, we have $\rho(AB)=\tau(ABd^2)=\tau(Bd^2A)=\tau(d^2BA)=\rho(BA)$ by $(3.2)$. Thus $\rho $ is a faithful normal trace on $\mathcal M$ such that $\rho\circ\Phi=\rho$ and $\mathfrak A$ is finite.
\end{proof}

\end{document}